\documentclass[a4paper,12pt]{article}

\usepackage[utf8x]{inputenc}
\usepackage{amsmath}
\usepackage{amsfonts}
\usepackage{amssymb}
\usepackage{graphicx}
\usepackage{amsthm}
\usepackage{varioref}
\usepackage{enumerate}

\newtheorem{theo}{Theorem}[section]
\newtheorem{lemma}[theo]{Lemma}
\newtheorem{claim}[theo]{Claim}
\newtheorem{prop}[theo]{Proposition}
\newtheorem{fact}[theo]{Facts}

\def\o{\overline}

\def\C{\mathcal{C}}
\def\F{\mathcal{F}}
\def\d{\hat{d}}
\def\X{{\sf X}}
\def\Y{{\sf Y}}
\def\Z{{\sf Z}}
\def\A{{\sf A}}
\def\B{{\sf B}}

\title{On $(2k,k)$-connected graphs}

\author{Olivier Durand de Gevigney\thanks {Laboratoire G-SCOP, CNRS, Grenoble INP, UJF,  46, Avenue F\'elix Viallet, Grenoble, France, 38000.} 
\thanks{olivier.durand-de-gevigney@g-scop.inpg.fr}\\ Zolt\'an Szigeti$^*$\thanks {Zoltan.Szigeti@g-scop.inpg.fr}}

\begin{document}

\maketitle

\abstract{A graph $G$ is called $(2k,k)$-connected if $G$ is $2k$-edge-connected and $G-v$ is $k$-edge-connected for every vertex $v$. The study of $(2k,k)$-connected graphs is motivated by a conjecture of Frank \cite{Frank_1995} which states that a graph has a $2$-vertex-connected orientation if and only if it is $(4,2)$-connected.
In this paper, we provide a construction of the family of $(2k,k)$-connected graphs for $k$ even which generalizes the construction given by Jord\'an \cite{Jordan_2006} for $k=2.$  
We also solve the corresponding connectivity augmentation problem: given a graph $G$ and an integer $k\geq 2$,  what is the minimum number of edges to be added to make $G$ $(2k,k)$-connected. 
Both these results are based on a new splitting-off theorem for $(2k,k)$-connected graphs. 
}

\section{Introduction}

Let $G=(V,E)$ be a graph and $k$ a positive integer. Loops and parallel edges are allowed in $G$. The graph $G$ is called \emph{$\ell$-edge-connected} if, for all $F \subseteq E$ such that $|F|<\ell$, $G-F$ is connected. The graph $G$ is called \emph{$\ell$-vertex-connected} if, $|V| > \ell$ and for all $X \subseteq V$ such that $|X|<\ell$, $G-X$ is connected. The graph $G$ is called \emph{$(2k,k)$-connected} if $|V| \geq 3$, $G$ is $2k$-edge-connected and, for all $v \in V$, $G-v$ is $k$-edge-connected. This connectivity is a special case of a mixed-connectivity introduced by Kaneko and Ota \cite{Kaneko_Ota2000} which contains both vertex-connectivity and edge-connectivity.
\medskip

To motivate our problems, let us start with a result on orientations of graphs. Nash-Williams \cite{NashWill} proved that an undirected  graph has a $k$-edge-connected orientation if and  only if it is $2k$-edge-connected. 

This theorem can easily be proved by applying Lov\'asz' construction \cite{LovaszSplit} of $2k$-edge-connected graphs. He proved that  a graph is $2k$-edge-connected if and only if it can be obtained from $2kK_2$, the graph on $2$ vertices  with $2k$ edges between them, by repeating the following two operations:  adding an edge and pinching $k$ edges, that is subdividing each of the $k$ edges by a new vertex and identifying these new vertices. 

To prove Lov\'asz' construction one has to consider the inverse operations: deleting an edge and complete splitting-off at a vertex of degree $2k.$ Let us now introduce the operation of complete splitting-off at a vertex $s$ of even degree. It consists of partitioning the set of edges incident to $s$ into pairs, replacing each pair $(su,sv)$ by a new edge $uv$ and then deleting $s.$ When no edge can be deleted without destroying $2k$-edge-connectivity, it is easy to prove that there exists a vertex of degree  $2k.$ Then  Lov\'asz' splitting-off theorem \cite {LovaszSplit} implies the existence of a complete splitting-off at this vertex that preserves $2k$-edge-connectivity.

We mention that Lov\'asz' splitting-off theorem is valid for $\ell$-edge-connectivity where $\ell$ is any integer larger than $1.$ This theorem  has other applications, among others, it can be used  to solve the $\ell$-edge-connected augmentation problem (Frank \cite{frank_1992}). This can be formulated as follows: given a graph $G$ and an integer  $\ell\geq 2$, what is the minimum number of edges to be added to make $G$ $\ell$-edge-connected.\\
\medskip

Inspired by Nash-Williams' result, Frank  \cite{Frank_1995} proposed a conjecture that would characterize undirected graphs having a $k$-vertex-connected orientation. We present here only the special case $k=2.$ He conjectured that an undirected graph has a $2$-vertex-connected orientation if and only if it is $(4,2)$-connected.

This conjecture drew attention on the family of $(4,2)$-connected graphs. Jord\'an \cite{Jordan_2006}  gave a construction of this family,  similar to Lov\'asz' construction of $2k$-edge-connected graphs. He showed that a graph is $(4,2)$-connected if and only if it can be obtained from $2K_3$, the graph on $3$ vertices and $2$ edges between each pair of vertices, by repeating the following two operations:  adding an edge and pinching $2$ edges such that if one of them is a loop then the other one is not adjacent to it. 
Unfortunately, this construction helped to prove Frank's conjecture only in the Eulerian case \cite{BergTibor}.
To get this construction Jord\'an proved a splitting-off theorem on $(4,2)$-connected graphs. Here it is possible that there exists no complete splitting-off preserving $(4,2)$-connectivity, in this case a special kind of obstacle exists.
\medskip

We propose here to study $(2k,k)$-connected graphs.
First, we provide  a new splitting-off theorem for $(2k,k)$-connected graphs. As in the special case $k=2$, the existence of a complete splitting-off preserving $(2k,k)$-connectivity depends on the non-existence of an obstacle. Second,  we give  a construction of the family of $(2k,k)$-connected graphs for $k$ even. These are the natural  generalizations of the previous  results of Jord\'an \cite{Jordan_2006} on (4,2)-connected graphs. Finally, we solve the $(2k,k)$-connectivity augmentation problem. We follow Frank's approach \cite{frank_1992}: we find a minimal extension and then we apply our splitting-off theorem. This way we provide a new polynomial case for  connectivity augmentation.


\section{Definitions}

Let $\Omega$ be a ground set. The \emph{complement} of a subset $U \subseteq \Omega$ is defined by $\o{U}=\Omega \setminus U$.
For $X_I \subseteq X_O \subseteq \Omega$, $\X=(X_O,X_I)$ is called a \emph{bi-set} of $\Omega$. The sets $X_I$, $X_O$ and  $w(\X) = X_O\setminus X_I$ are, respectively, the \emph{inner-set}, the \emph{outer-set} and the \emph{wall} of $\X$. If $X_I= \emptyset$ or $X_O = \Omega$, the bi-set $\X$ is called \emph{trivial}. If $w(\X)$ is non-empty then $\X$ is called a \emph{pair} and a \emph{set} otherwise.
The \emph{intersection} and the \emph{union} of two bi-sets $\X=(X_O,X_I)$ and $\Y=(Y_O,Y_I)$ are defined by $\X \sqcap \Y= (X_O \cap Y_O, X_I \cap Y_I)$ and $\X \sqcup \Y =(X_O \cup Y_O, X_I \cup Y_I)$. We say that \emph{$\X$ is included in $\Y$}, denoted by $\X \sqsubseteq \Y$, if $X_O \subseteq Y_O$ and $X_I \subseteq Y_I$. We say that $\X$ and $\Y$ are \emph{innerly-disjoint} if $X_I \cap Y_I=\emptyset$.
We extend the complement operation to bi-sets by defining the complement of $\X$ as $\o{\X}=(\o{X_I},\o{X_O})$. For a family $\F$ of bi-sets of $\Omega$, we denote $\Omega_I(\F) = \cup_{\X \in \F}X_I$.
A bi-set function $b$ is called \emph{submodular} if, for all bi-sets $\X$ and $\Y$,
\begin{equation}
 b(\X) + b(\Y) \geq b(\X \sqcap \Y) + b(\X \sqcup \Y).
\end{equation}

Let $G=(V,E)$ be a graph. For $U,W \subset V$, $d_G(U,W)$ 
denotes the number of edges with one end-vertex in $U\setminus W$ and the other end-vertex in $W \setminus U$. For short, $d_G(U,\o{U})$ is denoted by $d_G(U)$.
An edge $uv$ \emph{enters} a bi-set $\X=(X_O,X_I)$ of $V$, if $u \notin X_O$ and $v \in X_I$. The \emph{degree} of $\X$, $\hat{d}_G(\X)$, is the number of edges entering $\X$. Observe that $\d_G$ is symmetric with respect to the complement operation of bi-sets.\\

We can reformulate the $(2k,k)$-connectivity using bi-sets.
Note that, the graph $G$ is $(2k,k)$-connected, if $|V|\geq3$ and, for all non-trivial bi-sets $\X$ of $V$,
\begin{equation} \label{def_con}
 f_G(\X) := \hat{d}_G(\X) + k|w(\X)| \geq 2k.
\end{equation}
The graph $G$ is called \emph{minimally} $(2k,k)$-connected if $G$ is $(2k,k)$-connected and $G-e$ is not $(2k,k)$-connected for any $e \in E$.
Note that if $\X$ is a set, $f_G(\X) = \hat{d}_G(\X) = d_G(X_I) = d_G(X_O).$

Let $H=(V+s,E)$ be a graph with a special vertex $s$. For convenience, in this paper $H$ will always denote a graph with a special vertex $s$. 
We denote by $N_H(s)$ the set of neighbors of $s$ in $H$.
The graph $H$ is called \emph{$(2k,k)$-connected in $V$} if  $|V|\geq3$ and \eqref{def_con} holds in $H$ for all non-trivial bi-sets of $V$. A bi-set of $V+s$ satisfying \eqref{def_con} with equality is called \emph{tight}.
The graph $H$ is called \emph{$k$-edge-connected in $V$} if $d_H(X) \geq k$ for every non-trivial set $X$ of $V.$
Note that, considering the graph $H$, the complement of a set or a bi-set is taken relatively to the ground set $V+s$.

Let $(su,sv)$ a pair of (possibly parallel) edges. \emph{Splitting-off} the pair $(su,sv)$ at $s$ in $H$ consists of replacing the edges $su,sv$ by a new edge $uv$. The graph arising from this splitting-off at $s$ is denoted by $H_{u,v}$.
If $d_H(s)$ is even then a sequence of $\frac{1}{2}d_H(s)$ splitting-off of disjoint pairs at $s$ is called a \emph{complete splitting-off} at $s$.
If $H$ and $H_{u,v}$ are $(2k,k)$-connected in $V$, then the pair $(su,sv)$ is called $(2k,k)$-\emph{admissible} (shortly, \emph{admissible} when $k$ is clear from the context). A complete splitting-off is called \emph{admissible} if the resulting graph is $(2k,k)$-connected in $V$.

The inverse operation of a complete splitting-off is defined as follows. Let $G=(V,E)$ be a graph, for $F \subseteq E$, \emph{pinching} $F$, consists of adding a new vertex in the middle of each edge in $F$ and, then, identifying these new vertices as a single one.\\

\section{Preliminaries}

In this section we provides some basic observations.

\begin{claim} 
Let $H=(V+s,E)$ be a $(2k,k)$-connected graph in $V$ and $\X$  a non-trivial bi-set of $V$. Then
 \begin{eqnarray}
2k-f_H(\X) & \leq & d_H(s,\o{X_O}) - d_H(s,X_I),\label{plusineq}\\
 d_H(s,X_I) & \leq & \left \lfloor \frac{1}{2}(d_H(s)-d_H(s,w(\X))+f_H(\X)-2k) \right \rfloor.\label{neigbors_of_s}
\end{eqnarray}
\end{claim}

\begin{proof}  Note that $(\o{X_I}-s,\o{X_O}-s)$ is a non-trivial bi-set of $V$ and that $f_H(\o{X_I}-s,\o{X_O}-s) - f_H(\X) = d_H(s,\o{X_O}) - d_H(s,X_I)$. Then, by \eqref{def_con}, \eqref{plusineq} follows.
By \eqref{plusineq}, and $d_H(s,\o{X_O}) - d_H(s,X_I) = d_H(s) - d_H(s,w(\X)) - 2d_H(s,X_I)$, \eqref{neigbors_of_s} follows.
\end{proof}

\begin{prop} \label{submodularity}
 Let $H=(V+s,E)$ be a graph, and $\X$ and $\Y$ two bi-sets of $V+s$. We have
\begin{equation}
 \hat{d}_H(\X) + \hat{d}_H(\Y) = \hat{d}_H(\X \sqcap \Y) + \hat{d}_H(\X \sqcup \Y) + d_H(\o{X_O} \cap Y_O, X_I \cap \o{Y_I}) + d_H(\o{Y_O} \cap X_O, Y_I \cap \o{X_I}). \label{main_eq}
\end{equation}
 Moreover, if $H$ is $(2k,k)$-connected in $V$, $|w(\X \sqcup \Y)|\geq 2$ and $\X \sqcap \Y$ is a non-trivial bi-set of $V$, then
\begin{equation} \label{pairs}
 (f_H(\X)-2 k) + (f_H(\Y)-2 k) \geq \d_H(\X \sqcup \Y) + d_H(\o{X_O} \cap Y_O, X_I \cap \o{Y_I}) + d_H(\o{Y_O} \cap X_O, Y_I \cap \o{X_I}).  
\end{equation}
\end{prop}
\begin{proof}
 \eqref{main_eq} We let the reader carefully check that any edge participates to the same amount on both sides.
 \medskip

 \eqref{pairs} By modularity of $k|w(.)|$, by \eqref{def_con}, since $\X \sqcap \Y$ is a non-trivial bi-set of $V$, by $|w(\X \sqcup \Y)|\geq 2$ and by \eqref{main_eq}, we have, 
$(f_H(\X)-2 k) + (f_H(\Y)-2 k) 
  = \d_H(\X) + \d_H(\Y) + k|w(\X)| + k|w(\Y)| - 2k - 2k 
  \geq \d_H(\X) + \d_H(\Y) + k|w(\X \sqcap \Y)| + k|w(\X \sqcup \Y)| - f_H(\X \sqcap \Y) - 2k 
  \geq \d_H(\X) + \d_H(\Y) - \d_H(\X \sqcap \Y) 
  \geq \d_H(\X \sqcup \Y) + d_H(\o{X_O} \cap Y_O, X_I \cap \o{Y_I}) + d_H(\o{Y_O} \cap X_O, Y_I \cap \o{X_I}).$
\end{proof}
\noindent Note that, by \eqref{main_eq} and modularity of $w(.)$,  $f_H(.)$ is submodular.

\begin{claim}\label{tight_intersection}
 Let $H=(V+s,E)$ be a $(2k,k)$-connected graph in $V$ and $\X$ and $\Y$ two tight bi-sets of $V+s.$ If $\X \sqcap \Y$ and $\X \sqcup \Y$ (resp. $\o{\X \sqcup \Y}$) are non-trivial bi-sets of $V$ then $\X \sqcap \Y$ and $\X \sqcup \Y$ (resp. $\o{\X \sqcup \Y}$) are tight and 
 $d_H(\o{X_O} \cap Y_O, X_I \cap \o{Y_I}) = d_H(\o{Y_O} \cap X_O, Y_I \cap \o{X_I}) = 0.$
\end{claim}
\begin{proof}
 Let $\Z = \X \sqcup \Y$ (resp.  $\Z = \o{\X \sqcup \Y}$).
 By tightness of $\X$ and $\Y$, \eqref{main_eq},  modularity of $k|w(.)|$, non-negativity of the degree function $d_H$, symmetry of $f_H$ and $(2k,k)$-connectivity in $V$ of $H$, we have
 $ 2k + 2k = f_H(\X) + f_H(\Y) 
 = f_H(\X \sqcap \Y) + f_H(\X \sqcup \Y) + d_H(\o{X_O} \cap Y_O, X_I \cap \o{Y_I}) + d_H(\o{Y_O} \cap X_O, Y_I \cap \o{X_I})
  \geq  f_H(\X \sqcap \Y) + f_H(\Z) + 0 + 0 
  \geq  2k + 2k.$
 Hence there exists equality everywhere and the claim follows. 
\end{proof}

\section{Blocking bi-sets}

Let $H=(V+s,E)$ be a $(2k,k)$-connected graph in $V$ with a special vertex $s$ and $(su,sv)$ a pair of edges. A non-trivial bi-set $\X$ of $V$ is called a \emph{blocking bi-set} for the pair $(su,sv)$ if,
\begin{itemize}
  \item[(a)] either $u,v \in X_I$ and $f_H(\X) \leq 2k+1$,
  \item[(b)] or $u \in X_I$, $v = w(\X)$ (or $v \in X_I$, $u = w(\X)$) and $f_H(\X) \leq 2k.$
\end{itemize}
Note that by (a) and (b), for a blocking bi-set $\X$,
\begin{eqnarray}
 |w(\X)| & \leq &  1, \label{w_leq_1} \\
 f_H(\X) - 2k & \leq & d_H(s,X_I) -1.\label{blocking_ineq}
\end{eqnarray}
For short we say that $\X$ \emph{blocks} $(su,sv)$.
If (a) occurs, then $\X$ is called \emph{dangerous} and \emph{critical} otherwise. Note that critical pairs are tight. Note also that if $\X$ blocks $(su,sv)$ then, after any sequence of splitting-off not containing $su$ nor $sv$, $\X$ still blocks $(su,sv)$. The term blocking is justified by the following lemma.

\begin{lemma}\label{blocking_bi-set}
 Let $H=(V+s,E)$ be a $(2k,k)$-connected graph in $V$. A pair $(su,sv)$ is non-admissible if and only if there exists a bi-set of $V$ blocking $(su,sv)$.
\end{lemma}
\begin{proof} The sufficiency is clear. Let us see the necessity.
Since $(su,sv)$ is non-admissible, there exists a non-trivial bi-set $\X$ of $V$ which violates (\ref{def_con}) in $H_{u,v}$. Since $f_H(\X) \geq 2k$, either $\d_{H_{u,v}}(\X) = \d_H(\X) - 2$, that is $u,v \in X_I$ and $f_H(\X) \leq 2k+1$ or $\d_{H_{u,v}}(\X) = \d_H(\X) - 1$, that is $u \in X_I$ and $v = w(\X)$ (or $v \in X_I$ and $u = w(\X)$), and $f_H(\X) \leq 2k.$
\end{proof}

 By the remark above, if the pair $(su,sv)$ is non-admissible in $H$, then $(su,sv)$ is non-admissible in any graph arising from $H$ by a sequence of splitting-off. 
\medskip

We will heavily rely on the following lemma whose proof is quite technical.

\begin{lemma} \label{main_lemma}
 Let $H=(V+s,E)$ be a $(2k,k)$-connected graph in $V$ with $d_H(s)$ even. Let $\X$ be a maximal blocking bi-set for a pair $(su,sv)$ with $u \in X_I.$ Let $w \in N_H(s) \setminus X_I$  and $\Y$ a blocking bi-set for the pair  $(su,sw)$. Then $\X$ and $\Y$ are both pairs with the same wall.
\end{lemma}
\begin{proof} 
\begin{fact} $\X$ and $\Y$ satisfy the following. \label{facts}
 \begin{enumerate}[(a)]
  \item \label{X_cap_Y_not_trivial} 
  If $w(\Y) \cap X_I=\emptyset$ then $\X \sqcap \Y$ is a non-trivial bi-set of $V$.  
  \item \label{oX_cap_Y_not_trivial} If $w(\X) \cap Y_I = \emptyset = w(\Y) \cap \o{X_I}$ then $\o{\X} \sqcap \Y$ is a non-trivial bi-set of $V$.
  \item \label{X_cap_oY_not_trivial} If  $w(\X) \cap \o{Y_I}= \emptyset = w(\Y) \cap X_I$ then $\X \sqcap \o{\Y}$ is a non-trivial bi-set of $V$.
  \item \label{X_cup_Y_not_trivial} If $w(\X) \cap \o{Y_I}= \emptyset = w(\Y) \cap \o{X_I}$ then $\X \sqcup \Y$ is a non-trivial bi-set of $V$ strictly containing $\X$.
 \end{enumerate}
\end{fact}
 
\begin{proof}
\eqref{X_cap_Y_not_trivial} Since $u \in X_I$, $\Y$ blocks $(su,sw)$ and $w(\Y) \cap X_I=\emptyset$, we have $u \in X_I \cap Y_O = X_I \cap (Y_I \cup w(\Y))=X_I \cap Y_I$. Since $\X$ is non-trivial and $X_I \cap Y_I \neq \emptyset$, $\X \sqcap Y$ is non-trivial.
\medskip

\eqref{oX_cap_Y_not_trivial} Since $Y$ blocks $(su,sw)$, $w \notin X_I$, by $w(\Y) \cap \o{X_I}=\emptyset$ and $w(\X) \cap Y_I = \emptyset$, we have $w \in Y_O \setminus X_I = Y_O \cap \o{X_I} = (w(\Y)\cup Y_I)\cap \o{X_I} = Y_I \cap \o{X_I} = Y_I \cap (w(\X)\cup\o{X_O}) = Y_I \cap \o{X_O}.$ Thus, since $\Y$ is non-trivial, $\o{\X} \sqcap \Y$ is non-trivial.
\medskip

\eqref{X_cap_oY_not_trivial} Suppose that $\X \sqcap \o{\Y}$ is  trivial. Since  $\X$ is non-trivial,  $X_I \cap \o{Y_O}= \emptyset$, that is $X_I \subseteq Y_O =w(\Y) \cup Y_I$. By $w(\Y) \cap X_I = \emptyset$, $X_I \subseteq Y_I.$  Then, by $w(\X) \cap \o{Y_I}= \emptyset$, $w(\X)\subseteq Y_I$ and $X_O=X_I\cup w(\X)\subseteq Y_I \subseteq Y_O.$ By maximality of $\X$, $\X = \Y$ and $w(Y)=\emptyset$. Since $Y$ blocks $(su,sw)$, $w \in Y_O = Y_I = X_I,$ a contradiction to the choice of $w$. 
\medskip

\eqref{X_cup_Y_not_trivial}
Since $\Y$ blocks $(su,sw)$, by the choice of $w$ and $w(\Y) \cap \o{X_I}=\emptyset$, we have $w \in Y_O \cap \o{X_I} = Y_I \cap \o{X_I}.$ Hence $\X \sqcup \Y$ strictly contains $\X$.
Since  $\X$ is non-trivial and by the conditions, it remains to prove that $V \neq X_O \cup Y_O = X_I \cup Y_I.$ We have $d_H(s,X_I \cup Y_I) = d_H(s,X_I) + d_H(s,Y_I) - d_H(s,X_I \cap Y_I).$ 
If $\Y$ is critical then, by \eqref{neigbors_of_s} and $d_H(s)$ even, $d_H(s,X_I)+ d_H(s,Y_I)\leq \frac{1}{2}d_H(s) +\frac{1}{2}d_H(s)-1<d_H(s,V).$
If $\Y$ is dangerous then $u \in X_I \cap Y_I\cap N_H(s)$, hence by \eqref{neigbors_of_s} and $d_H(s)$ even, $d_H(s,X_I) + d_H(s,Y_I) - d_H(s,X_I \cap Y_I) \leq \frac{1}{2}d_H(s) + \frac{1}{2}d_H(s)-1 < d_H(s,V).$
 In both cases we have $X_I \cup Y_I \neq V.$
\end{proof}
\begin{claim} \label{X_Y_not_both_pair}
 $\X$ and $\Y$ are not both sets.
\end{claim}
\begin{proof} Suppose $\X$ and $\Y$ are both sets. By Facts \eqref{X_cap_Y_not_trivial}, \eqref{X_cap_oY_not_trivial}, \eqref{oX_cap_Y_not_trivial} and \eqref{X_cup_Y_not_trivial}, we have $u \in Y_I \cap X_I \cap N_H(s)$ and $\X \sqcap \Y$, $\X \sqcap \o{\Y}$, $\o{\X} \sqcap \Y$ and $\X \sqcup \Y$ are non-trivial and $\X \sqsubset\X \sqcup \Y$. Hence by (\ref{main_eq}), $(2k,k)$-connectivity of $H$ and maximality of $\X$, we have $(2k+1) + (2k+1) \geq \hat{d}_H(\X) + \hat{d}_H(\o{\Y}) \geq \hat{d}_H(\X \sqcap \o{\Y}) + \hat{d}_H(\o{\X} \sqcap \Y) + 2d_H(s,X_I \cap Y_I) \geq 2k + 2k + 2$ and $(2k+1) + (2k+1) \geq \hat{d}_H(\X) + \hat{d}_H(\Y) \geq \hat{d}_H(\X \sqcap \Y) + \hat{d}_H(\X \sqcup \Y) \geq 2k + (2k + 2).$ It follows that equality holds everywhere, in particular, $\d_H(\X \sqcap \Y)=\d_H(\X \sqcap  \o{\Y}) = 2k$ are even and $\d_H(\X) = 2k+1$ is odd. This contradicts $\d_H(\X) = \d_H(\X \sqcap \Y) + \d_H(\X \sqcap \o{\Y}) -2d_H(X_I \cap Y_I, X_I \cap \o{Y_I}).$
\end{proof}
\begin{claim} \label{X_Y_pairs}
 $\X$ and $\Y$ are both pairs.
\end{claim}
\begin{proof}
Suppose $\X$ or $\Y$ is a set, call it $\A$, by Claim \ref{X_Y_not_both_pair}, the other blocking bi-set is a pair, call it $\B$.

Suppose $w(\B) \cap A_I = \emptyset$. Then, since $\A$ is a set, $\A \sqcap \B$ and $\A \sqcap \o{\B}$ are sets. By Facts \ref{facts} \eqref{X_cap_Y_not_trivial} and \eqref{X_cap_oY_not_trivial} or \eqref{oX_cap_Y_not_trivial}, $\A \sqcap \B$ and $\A \sqcap \o{\B}$ are non-trivial. Then, by $(2k,k)$-connectivity of $H$ in $V$, since the edges between $A_I \setminus B_I$ and $A_I \cap B_I$ enters $\B$ but not $s$ and by \eqref{blocking_ineq}, we have the following contradiction, 
$2k + 2k  \leq  f_H(\A \sqcap \B) + f_H(\A \sqcap \o{\B}) 
= \hat{d}_H(\A \sqcap \B) + \hat{d}_H(\A \sqcap \o{\B})
= \d_H(\A) + 2d_H(A_I \setminus B_I, A_I \cap B_I)
\leq f_H(\A)  + 2(\hat{d}_H(\B)-d_H(s,B_I))
= f_H(\A)  + 2(f_H(\B)-k|w(\B)|-d_H(s,B_I)) 
\leq 2k+1 +2(k-1).$

Hence, by \eqref{w_leq_1} for $\B$, we have $w(\B) \cap \o{A_I} = \emptyset$. Then, since $\A$ is a set, $\A \sqcup \B$ and $\o{\A} \sqcap \B$ are sets. By Fact \ref{facts} \eqref{X_cup_Y_not_trivial} and \eqref{oX_cap_Y_not_trivial} or \eqref{X_cap_oY_not_trivial}, $\A \sqcup \B$ and $\o{\A} \sqcap \B$ are non-trivial and $\X \sqsubset\A \sqcup \B$. We have ($\star$) $\d_H(\B)-d_H(s,A_I \cap B_I) \leq k$ since $\B$ is a blocking pair and, if $\B$ is dangerous then, $u \in A_I \cap B_I$.  By  maximality of $\X$, the $(2k,k)$-connectivity of $H$, since $\A$ is a blocking set and since the edges between $\o{A_I \cup B_I}$ and $B_I \setminus A_I$ enters $\B$ but not $A_I \cap B_I$, by \eqref{blocking_ineq} and ($\star$), we have the following contradiction,
$(2k+2)+2k
  \leq  f_H(\A \sqcup \B) + f_H(\o{\A} \sqcap \B)
  =  d_H(A_I \cup B_I) + d_H(B_I \setminus A_I)
  =  d_H(\o{A_I \cup B_I}) + d_H(B_I \setminus A_I)
  =  d_H(\o{A_I}) + 2d_H(\o{A_I \cup B_I},B_I \setminus A_I)
  \leq f_H(\A) +2(\d_H(\B) - d_H(s,A_I \cap B_I))
  \leq (2k+1) + 2k.$
\end{proof}
\begin{claim}\label{wall}
 $\Y$ and $\X$ have the same wall.
\end{claim}
\begin{proof} Suppose $w(\X) \neq w(\Y)$. By Claim \ref{X_Y_pairs} and \eqref{w_leq_1}, both $w(\X)$ and $w(\Y)$ are singletons, we have 4 cases.
  \medskip

  \textbf{Case 1} If $w(\X) \cap Y_I = \emptyset = w(\Y) \cap X_I$. Then $|w(\X \sqcup \Y)|  = 2$. By Fact \eqref{X_cap_Y_not_trivial}, $\X\sqcap \Y$ is a non-trivial bi-set of $V$. Hence, by \eqref{pairs}, since $\X$ and $\Y$ are blocking bi-sets, by \eqref{blocking_ineq} and the choice of $w$, we have the following contradiction, $\d_H(\X \sqcup \Y) \leq (f_H(\X) -2k) + (f_H(\Y)-2k) \leq (d_H(s,X_I)-1) + d_H(s,Y_I \setminus X_I) = d_H(s,X_I \cup Y_I) - 1 \leq \d_H(\X \sqcup \Y) -1.$
  \medskip

  \textbf{Case 2} If $w(\X) \cap Y_I = \emptyset = w(\Y) \cap \o{X_I}$. Then $|w(\o{\X} \sqcup \Y)|=2$.  By Fact \eqref{oX_cap_Y_not_trivial}, $\o{\X} \sqcap \Y$ is a non-trivial bi-set of $V$. By symmetry of $f_H$ and \eqref{blocking_ineq}, $f_H(\o{\X}) -2k = f_H(\X) - 2k < d_H(s,X_I) \leq \d_H(\o{\X} \sqcup \Y) + d_H(X_I \cap Y_O,\o{X_O} \cap \o{Y_O}).$ If $\Y$ is dangerous, then $u \in X_I \cap Y_I$. Hence $f_H(\Y)-2k \leq d(s,X_I \cap Y_I) \leq d_H(\o{Y_O} \cap \o{X_I},Y_I \cap X_O).$ So we have, $(f_H(\o{\X}) -2k)+(f_H(\Y) -2k)<\d_H(\o{\X} \sqcup \Y) + d_H(X_I \cap Y_O,\o{X_O} \cap \o{Y_O}) + d_H(\o{Y_O} \cap \o{X_I},Y_I \cap X_O)$ and this contradicts \eqref{pairs}.
  \medskip

  \textbf{Case 3} If $w(\X) \cap \o{Y_I} = \emptyset = w(\Y) \cap X_I$. Then $|w(\X \sqcup \o{\Y})|=2$. By Fact \eqref{X_cap_oY_not_trivial}, $\X \sqcap \o{\Y}$ is a non-trivial bi-set of $V$. By symmetry of $f_H$ and \eqref{blocking_ineq}, $f_H(\o{\Y})-2k = f_H(\Y)-2k < d_H(s,Y_I) \leq \d_H(\X \sqcup \o{\Y}) + d_H(Y_I \cap X_O, \o{Y_O} \cap \o{X_I}).$ Since $w(\Y) \cap X_I = \emptyset$, we have $u \in Y_I \cap X_I.$ Hence $f_H(\X) - 2k \leq d(s,X_I \cap Y_I) \leq d_H(\o{X_O} \cap \o{Y_I},X_I \cap Y_O).$ So we have, $(f_H(\X) -2k)+(f_H(\o{\Y}) -2k)<\d_H(\X \sqcup \o{\Y}) + d_H(\o{X_O} \cap \o{Y_I},X_I \cap Y_O) + d_H(Y_I \cap X_O, \o{Y_O} \cap \o{X_I})$ and this contradicts \eqref{pairs}.
  \medskip

  \textbf{Case 4} If $w(\X) \cap \o{Y_I} = \emptyset = w(\Y) \cap \o{X_I}$. Then $|w(\X \sqcap \Y)|=2$. By Fact \eqref{X_cup_Y_not_trivial}, $\X \sqcup \Y$ is a non-trivial bi-set of $V$ and $\X \sqsubset\X \sqcup \Y$. If $\Y$ is dangerous, then $u \in X_I \cap Y_I$, thus, since $\X$ is a blocking bi-set, $1+\d_H(\X \sqcap \Y) \geq 1+d_H(s,X_I \cap Y_I) \geq (f_H(\X)-2k) + (f_H(\Y)-2k)$. By maximality of $\X$ and submodularity of $f_H$, we have the following contradiction, $2k+2 \leq f_H(\X \sqcup \Y) \leq f_H(\X) + f_H(\Y) - f_H(\X \sqcap \Y) \leq \d_H(\X \sqcap \Y) + 1 + 4k - f_H(\X \sqcap \Y) = 1 + 2k.$\end{proof} 
  
  Claims \ref{X_Y_pairs} and \ref{wall} prove Lemma \ref{main_lemma}.
  \end{proof}

\begin{prop} \label{non_trivial}
 Let $H=(V+s,E)$ be a $(2k,k)$-connected graph in $V$ with $d_H(s)$ even and
 $\X$ and $\Y$ two critical pairs with wall $\{w\}$ such that $d_H(s,w)$ is odd. Then $N_H(s) \setminus (X_O \cup Y_O)$ is non-empty.
 In particular, $\X \sqcup \Y$ is a non-trivial bi-set of $V$.
\end{prop}
\begin{proof}
 By \eqref{neigbors_of_s}, $d_H(s)$ even and $d_H(s,w)$ odd, we have 
$d_H(s,X_O \cup Y_O) = d_H(s,X_I \cup Y_I) + d_H(s,w) \leq d_H(s,X_I) + d_H(s,Y_I) + d_H(s,w) < \frac{1}{2}\big(d_H(s) - d_H(s,w)\big) + \frac{1}{2}\big(d_H(s) - d_H(s,w)\big) + d_H(s,w) = d_H(s).$ 
Hence,  there exists a neighbor of $s$ in $V \setminus (X_O \cup Y_O)$ that is $\X \sqcup \Y$ is non-trivial.
\end{proof}

\begin{claim} \label{neigbors_of_s3}
 Let $H=(V+s,E)$ be a a graph $(2k,k)$-connected in $V$ with $d_H(s)$ even.
 Let $\X$ be a maximal blocking bi-set for $(su,su)$ where $u \in V$ such that $d_H(s,u) \geq \frac{d_H(s)}{2}$. Then the pair $(su,sv)$ is splittable for all $v \in N_H(s) \setminus X_O$.
\end{claim}
\begin{proof}
Since $\X$ is obviously dangerous and $v \in N_H(s) \setminus X_O$, $w(\X) \cap \{u,v\}=\emptyset$.
 Suppose that $(su,sv)$ is non-admissible, that is by Lemma \ref{blocking_bi-set}, there exists a bi-set $\Y$ blocking the pair $(su,sv)$.  Hence, since $\X$ and $\Y$ are both pairs with the same wall by Lemma \ref{main_lemma}, $v,u \in Y_I$. This gives $d_H(s,Y_I) \geq d_H(s,u) + d_H(s,v) \geq \frac{d_H(s)}{2} + 1$. This contradicts \eqref{neigbors_of_s}.
\end{proof}

\section{Obstacles}

Let $H=(V+s,E)$ be a $(2k,k)$-connected graph in $V$ such that $d_H(s)$ is even. We extend the definition of Jord\'an \cite{Jordan_2006} as follows. The pair $(t,\C)$ is called a \emph{$t$-star obstacle} at $s$ (for short, an \emph{obstacle}) if
\begin{eqnarray}
 & & t \textrm{ is a neighbor of $s$ with } d_H(s,t) \textrm{ odd},\label{eq1}\\
 & & \C \textrm{ is a collection of critical pairs,}\label{eq5}\\
 & & \textrm{each element of $\C$ has wall } \{t\},\label{eq2}\\
 & & \textrm{the elements of }\C\textrm{ are pairwise innerly-disjoint,}\label{eq3}\\
 & & N_H(s)\setminus\{t\} \subseteq V_I(\C).\label{eq4}
\end{eqnarray}
If $(t,\C)$ is an obstacle at $s$, note that, by Lemma \ref{blocking_bi-set}, no pair $(st,su)$ with $u \in N_H(s) - t$ is admissible. Some basic properties of obstacles are proven in the following proposition.

\begin{prop}
 Let $H=(V+s,E)$ be a $(2k,k)$-connected graph in $V$ with $d_H(s)$ even and $(t,\C)$ an obstacle at $s$. Then
\begin{eqnarray}
 & & |\C| \geq 3,\label{C_geq_3}\\
 & & H-st \textrm{ is $(2k,k)$-connected in $V$.} \label{st_removable}
\end{eqnarray}
\end{prop}
%
\begin{proof}
 \eqref{C_geq_3}: By \eqref{eq4}, \eqref{eq1} and $d_H(s)$ even, $|\C| \geq 1$. 
 Let $\X$ and $\Y$ be two (not necessarily distinct) elements of $\C$. By \eqref{eq5}, \eqref{eq2}, \eqref{eq1} and Proposition \ref{non_trivial}, $N_H(s) \setminus (X_O \cup Y_O)$ is non-empty. Thus, by \eqref{eq4}, there exists an element in $\C \setminus \{\X,\Y\}$.
\medskip

 \eqref{st_removable}: Suppose that $H-st$ is not $(2k,k)$-connected in $V$, that is, by $(2k,k)$-connectivity of $H$, there exists in $H$ a non-trivial tight bi-set $\X$ of $V$ such that $t \in X_I$. Note that, by \eqref{eq1}, $|w(\X)|\leq1.$

 For all $\Y \in \C$, since $H$ is $(2k,k)$-connected in $V$ and $\Y$ is critical, $d_H(t,Y_I) = d_H(Y_I) - (f_H(Y_I)-k|w(\Y)|) \geq 2k - (2k-k) =k$. If $X_I = \{t\}$ then, by tightness of $\X$, \eqref{eq3}, \eqref{C_geq_3}, \eqref{eq1} and $|w(\X)|\leq1$, we have the following contradiction
$ 2k = f_H(\X)
 = \d_H(\X) + k|w(\X)|
 \geq \d_H(\X)
 = d_H(X_I) - d_H(X_I,w(\X))
 = d_H(t) - d_H(t,w(\X))
 \geq d_H(t,s) + \sum_{Y \in \C, w(\X) \notin Y_I}d_H(t,Y_I)
 \geq 1 + 2k.$
So $X_I \neq \{t\}$.

 Suppose that there exists $\Y \in \C$ such that $\o{X} \sqcap \Y$ and $\X \sqcap \o{\Y}$ are both non-trivial bi-sets of $V$. Then, since $\X$ is tight, $\Y$ is critical, by symmetry of $f_H$, \eqref{main_eq} and $(2k,k)$-connectivity of $H$ in $V$ and  Claim \ref{tight_intersection}, we have the following contradiction, $0=d_H(X_I \cap Y_O,\o{X_O} \cap \o{Y_I})\geq d_H(s,t)\geq 1.$
 Hence, for all $\Y \in \C$, $\o{X} \sqcap \Y$ or $\X \sqcap \o{\Y}$ is trivial, that is, since $\X$ and $\Y$ are non-trivial, $Y_I \subseteq X_O$ or $X_I \subseteq Y_O$.

If, for all $\Y \in \C$, $Y_I \subseteq X_O$ then, by $t \in X_I$ and \eqref{eq4}, $N_H(s) \subseteq X_I$.  This, by the tightness of $\X$,  contradicts \eqref{neigbors_of_s}.
So there exists $\Y \in \C$ such that $X_I \subseteq Y_O.$
 By \eqref{C_geq_3}, there exist at least two distinct elements $\A,\B \in \C\setminus \Y.$
 Since $X_I\neq \{t\}$, $X_I \subseteq Y_O$ and \eqref{eq3}, we have 
 $A_I,B_I \subseteq (X_O \setminus \{t\}) \cap \o{Y_I} = (X_I \cup w(\X)) \cap \o{Y_O} \subseteq w(\X),$ a contradiction to $|w(\X)| \leq 1$. 
\end{proof}

The following lemma shows that to find an obstacle one does not have to focus on the disjointness of the inner-sets.

\begin{lemma} \label{uncrossing}
 Let $H=(V+s,E)$ be a $(2k,k)$-connected graph in $V$ with $d_H(s)$ even. If there exists a pair $(t,\F)$ satisfying \eqref{eq1}, \eqref{eq5}, \eqref{eq2} and \eqref{eq4} then there exists a $t$-star obstacle at $s$.
\end{lemma}
\begin{proof}
  The proof applies the uncrossing method. Choose a pair $(t,\C)$ satisfying \eqref{eq1}, \eqref{eq5}, \eqref{eq2} and \eqref{eq4} such that $\sum_{\X \in \C}|X_I|$ is minimal.
  Suppose there exist two distinct elements $\X$ and $\Y$  in $\C$ such that $X_I \cap Y_I \neq \emptyset$ that is $\X \sqcap \Y$ is a non-trivial bi-set of $V$.
  By choice of $\C$, $\X \sqsubseteq \Y$ or $\Y \sqsubseteq \X$ is not possible. Hence, by \eqref{eq2}, $\X \sqcap \o{\Y}$ and $\o{\X} \sqcap \Y$ are non-trivial bi-sets of $V$. By Proposition \ref{non_trivial}, $\X \sqcup \Y$ is a non-trivial bi-set of $V$.

  Note that critical pairs are tight non-trivial bi-sets of $V$.
  Hence, by Claim \ref{tight_intersection}, $\X \sqcap \Y$, $\X \sqcap \o{\Y}$ and $\o{\X} \sqcap \Y$ are tight. The bi-sets among them which contain a neighbor of $s$ are critical pairs with wall $t$. Hence they can replace $\X$ and $\Y$ in $\C$ contradicting the minimality of $\sum_{\X \in \C}|X_I|$.
\end{proof}

\section{A new splitting-off theorem}

The first result of this section concerns the case when no admissible splitting-off exists. 

\begin{theo} \label{firstobstacle}
 Let $H=(V+s,E)$ be a graph that is $(2k,k)$-connected in $V$ with $d_H(s)$ even and $k \geq 2$. If there exists no admissible splitting-off at $s$ then $d_H(s)=4$ and there exists an obstacle at $s$.
\end{theo}
\begin{proof}
Suppose that there exists no admissible splitting-off at $s$, that is, by Lemma \ref{blocking_bi-set}, for each pair of edges incident to $s$, there exists a bi-set that blocks it.

Let $\X$ be a maximal blocking bi-set for a pair $(su,sv)$ with $u \in X_I.$ By \eqref{plusineq}, there exists a neighbor $w$ of $s$ in $\o{X_O} \subseteq \o{X_I}.$ Let $\Y$ be a maximal blocking bi-set for the pair $(su,sw)$. By Lemma \ref{main_lemma}, $\X$ and $\Y$ are pairs such that $w(\X)=w(\Y).$ Hence, by choice of $u$ and $w$, $\{w,u\} \cap w(\Y) = \{w,u\} \cap w(\X) = \emptyset$, that is $\Y$ is dangerous. By \eqref{w_leq_1}, $w(\Y)$ is a singleton, let us denote it by $\{t\}.$

For the same reasons, every maximal blocking bi-set for a pair $(sa,sb)$ with $a \in Y_I$ and $b \in N_H(s) \cap \o{Y_O} \neq \emptyset$ is a dangerous pair with wall $\{t\}.$
By repeating this argument one more time, we have that every pair $(sa,sb)$ with $a,b \notin \{t\}$ is blocked by a dangerous pair with wall $\{t\}.$
Hence, there exists a family $\mathcal{F}$ of (maximal) dangerous pairs such that \eqref{eq2} holds for $\mathcal{F}$ and every pair of edges adjacent to $s$ but not $t$ is blocked by an element of $\mathcal{F}.$

Now consider the graph $H-t$ which is, by $(2k,k)$-connectivity in $V$ of $H$, $k$-edge-connected in $V-t$. 
If $(su',sv')$ is a pair of edges in $H-t$ then, by the definition of $\mathcal{F}$, there exists a dangerous pair $\Z \in \mathcal{F}$ such that $u',v' \in \Z_I$ and $w(\Z)=\{t\}.$ Hence $d_{{(H-t)}_{u',v'}}(Z_I) = \d_{H_{u',v'}}(\Z) = \d_H(\Z) - 2 \leq f_H(\Z) - k|w(\Z)| - 2 \leq k-1$, that is splitting-off the pair $(su',sv')$ detroys the $k$-edge-connectivity in $V-t$ of $H-t.$

Hence, since $k \geq 2$, by a theorem of Mader \cite{MaderSplit}, $d_{H-t}(s)=3$. So, by $d_H(s)$ even and Claim \ref{neigbors_of_s3}, $d_H(s,t)$ is odd and smaller than $\frac{d_H(s)}{2}$ . That is $d_H(s,t)=1$ and $d_H(s)=4$.
Hence, by \eqref{plusineq}, the inner-set of each element of $\mathcal{F}$ contains exactly two neighbors of $s$ and $|\mathcal{F}|=3$.
So, for $\X \in \F$, $\X'=(\o{X_I}-s,\o{X_O}-s)$ is a non-trivial bi-set of $V$ and $X'_I$ contains exactly one neighbor of $s$, say $x$. We have $f_H(\X') = f_H(\X) - d_H(s,X_I) + d_H(s,V \setminus X_O) \leq 2k+1-2+1 = 2k$ thus $\X'$ is a critical pair blocking $(st,sx)$. So $(t,\F') = (t,\{\X': \X \in \F\})$ satisfies \eqref{eq1}, \eqref{eq5}, \eqref{eq2} and \eqref{eq4}. The obstacle at $s$ is obtained by applying Lemma \ref{uncrossing} on $(t,\F')$.
\end{proof}

The following lemma concerns the case when an obstacle occurs after an admissible splitting-off.

\begin{lemma} \label{pinching_up}
 Let $H=(V+s,E)$ be a $(2k,k)$-connected graph in $V$ with $d_H(s) \geq 6$ even, $(su,sv)$ an admissible pair in $H$ and $(t,\C)$ an obstacle at $s$ in $H_{u,v}$.
 \begin{itemize}
  \item[(a)] If $t \in \{u,v\}$ then $d_H(s,t) \geq 2$ and $(st,st)$ is admissible in $H$.
  \item[(b)] If $t \notin \{u,v\}$, either there exists a $t$-star obstacle at $s$ in $H$ or there exists no obstacle at $s$ in $H_{t,w}$ for some admissible pair $(st,sw)$ in $H$.
 \end{itemize}
\end{lemma}

\begin{proof} (a) Suppose $t = u$. By (\ref{eq1}) in $H_{u,v}$, $d_H(s,t) = d_{H_{u,v}}(s,t) + 1 \geq 2$.
 
Suppose now that $(st,st)$ is non-admissible in $H$. Then, by Lemma \ref{blocking_bi-set}, there exists in $H$ a maximal blocking bi-set $\X$ for this pair. 
If $v$ belongs to the inner-set of an element of $\C$ denote by $\Y$ this element and let $\Y=(\emptyset,\emptyset)$ otherwise. Since $u \in X_I$, $\X$ is a blocking bi-set, $\Y$ is a critical pair or empty, by \eqref{neigbors_of_s} and $d_H(s)$ even, we have,
$d_{H_{u,v}}(s,X_I \cup Y_I) \leq d_{H_{u,v}}(s,X_I) + d_{H_{u,v}}(s,Y_I)
\leq (d_H(s,X_I)-1) + d_{H_{u,v}}(s,Y_I)
\leq (\frac{1}{2}d_H(s) -1) + (\frac{1}{2}d_{H_{u,v}}(s) -1)
= d_{H_{u,v}}(s) -1.$
So, by \eqref{eq4}, there exists $\Z \in \C \setminus \Y$ and $w \in N_{H_{u,v}}(s) \setminus (X_I \cup \{v\})$, such that $w \in Z_I$ and $v  \notin Z_I$.
Hence, $\Z$ is also a blocking bi-set for $(st,sw)$ in $H$. Then, by Lemma \ref{main_lemma} applied in $H$, $w(\X)=w(\Z)=\{t\}$ which contradicts the fact that $\X$ blocks the splitting $(st,st)$  in $H$.\\

\noindent (b)
\begin{claim} \label{t_star_in_G}
If $st$ belongs to no admissible pair in $H$ then there exists a $t$-star obstacle in $H$. 
\end{claim}
\begin{proof}
By $t \notin \{u,v\}$ and (\ref{eq1}), $d_{H}(s,t) = d_{H_{u,v}}(s,t)$ is odd thus it remains to construct a collection of critical pairs satisfying (\ref{eq2}), (\ref{eq3}) and (\ref{eq4}). By Lemma \ref{uncrossing}, it suffices to find one satisfying (\ref{eq2}) and (\ref{eq4}).

We initialize $\F$ as $\{\X \in \C, |X_I \cap \{u,v\}| < 2\}$. $\F$ is a collection of critical pairs satisfying (\ref{eq2}). Suppose $\F$ does not satisfy (\ref{eq4}), that is, there exists $w \in N_H(s) \setminus (V_I(\F) \cup \{t\})$. Since $st$ belongs to no admissible pair,  by Lemma \ref{blocking_bi-set}, there exists a maximal blocking bi-set $\X$ for the pair $(st,sw)$. We prove that $w(\X)=\{t\}$ that is $\X$ can be added to the collection $\F$ constructed so far.

Assume, by contradiction, that $t \in X_I$. We have $N_H(s) \cap V_I(\F) \subseteq X_I$ otherwise, there exists $\Z \in \F$ such that $(N_H(s) \cap Z_I) \setminus X_I \neq \emptyset$, thus by Lemma \ref{main_lemma}, $w(\X)=w(\Z)=\{t\}$, a contradiction.
Thus, by $t \in X_I$, \eqref{neigbors_of_s}, $d_H(s)$ even and $d_H(s) \geq 6$, we have, 
$d_{H_{u,v}}(s) - d_{H_{u,v}}(s,V_I(\F) \cup \{t\})
\geq d_{H_{u,v}}(s) - d_{H_{u,v}}(s,X_I) 
\geq d_H(s) - 2 - d_H(s,X_I)
\geq \frac{1}{2}d_H(s) - 2
\geq 1.$
Hence, by (\ref{eq4}) in $H_{u,v}$, there exists a unique $\Y \in \C \setminus \F$ such that 
$N_{H_{u,v}}(s)\setminus\{t\} \subseteq V_I(\F) \cup Y_I$. 
Since $\{u,v\} \subseteq Y_I$, we have $N_H(s)\setminus \{t\} \subseteq V_I(\F) \cup Y_I$ and, in particular, $w \in Y_I$.
If $\X$ is dangerous, $w \in X_I \cap Y_I$ and, by \eqref{neigbors_of_s}, we have $d_H(s,X_I) - d_H(s,X_I \cap Y_I) \leq \frac{1}{2} d_H(s)-1.$ If $\X$ is critical, by \eqref{neigbors_of_s}, we have $d_H(s,X_I) \leq \frac{1}{2} d_H(s)-1.$ 
Hence, by $N_H(s) \subseteq X_I \cup Y_I$, and by \eqref{neigbors_of_s}, we have the contradiction,
$d_H(s) = d_H(s,Y_I) + d_H(s,X_I) - d_H(s,X_I \cap Y_I)
  = d_{H_{u,v}}(s,Y_I)+ 2 + d_H(s,X_I) - d_H(s,X_I \cap Y_I)
  \leq (\frac{1}{2}(d_H(s)-2)-1 + 2) + (\frac{1}{2}d_H(s) - 1) 
  = d_H(s) -1.$
\end{proof}
\begin{claim} \label{t_equal_t}
If $(st,sw)$ is an admissible pair in $H$ and there exists an obstacle $(t',\C')$ in $H_{t,w}$ then $t=t'$.
\end{claim}
\begin{proof}
 Suppose $t \neq t'$.
\begin{prop} \label{t_or_t}
Let $\X \in \C$ and $\X' \in \C'$ such that $\X \sqcap \X'$ is a non-trivial bi-set of $V$. Then $t \in X'_I$ or $t' \in X_I$.
\end{prop}
\begin{proof}
By contradiction assume that $t \notin X'_I$ and $t' \notin X_I$.
Thus, by $t \neq t'$ and $\X'$ critical in $H_{t,w}$, we have $t \notin X'_O$ and $f_H(\X')=f_{H_{t,w}}(\X')=2k$. 
Since $\X$ is critical in $H_{u,v}$ and by \eqref{blocking_ineq}, we have, $f_H(\X) - \d_H(\X \sqcup \X') \leq f_H(\X) - d_H(s,X_I) \leq f_{H_{u,v}}(\X) - d_{H_{u,v}}(s,X_I) \leq 2k-1$. 
Hence, since $|w(\X \sqcup \X')| = |\{t,t'\}| =2$ and $\X \sqcap \X'$ is non-trivial, by \eqref{pairs}, $0 \leq (f_H(\X)-2k) + (f_H(\X')-2k) - \d_H(\X \sqcup \X') \leq -1,$ a contradiction.
\end{proof}

\begin{prop} \label{t'_in_X}
 There exists $\X \in \C$ such that $t' \in X_I$.
\end{prop}
\begin{proof}
Suppose for a contradiction that $t' \notin V_I(\C)$. Then, by \eqref{eq1} for $(t',\C')$ in $H_{t,w}$ and  \eqref{eq4} for $(t,\C)$ in $H_{u,v}$ and $t \neq t'$, we have $t' \in \{u,v\}$,  say $t'=u$. 
By \eqref{C_geq_3} and \eqref{eq3} for $(t',\C')$ in $H_{t,w}$, there exists an element $\X' \in \C'$ containing neither $t$ nor $v$. Hence, there exists a vertex in $(N_H(s) \cap X_I) \setminus \{t',u,v,t\}$ which, by \eqref{eq4} for $(t,\C)$ in $H_{u,v}$, is contained in the inner-set of an element $\X \in \C$. Thus, $\X' \sqcap \X$ is a non-trivial bi-set of $V$, $t \notin X'_I$ and $t' \notin X_I$, a contradiction to Proposition \ref{t_or_t}.
\end{proof}

By Proposition \ref{t'_in_X}, there exists $\X \in \C$ such that $t' \in X_I$. 
By \eqref{C_geq_3} and \eqref{eq3} for $(t,\C)$ in $H_{u,v}$, there exists an element $\Y \in \C \setminus \{\X\}$ not containing $w.$ 
Hence, there exists a vertex in $(N_H(s) \cap Y_I) \setminus \{t',w,t\}$ which, by \eqref{eq4} in $H_{t,w}$, is contained in the inner-set of an element $\X' \in \C'$. Thus $\Y \sqcap \X'$ is non-trivial, so by Proposition \ref{t_or_t} and $t' \notin Y_I$, we have $t \in X'_I.$

Suppose that there exists a neighbor $z$ of $s$ in $H'=H-\{su,sv,sw\}$ that doesn't belong to $X_I$ nor $X'_I$. 
Then, by \eqref{C_geq_3}, \eqref{eq3}, $t' \in X_I$ and $t \in X'_I,$ there exists $\Z \in \C \setminus \X$ and $\Z' \in \C' \setminus \X'$ such that $z \in Z_I \cap Z'_I$. By $t' \in X_I$, $t \in X'_I$ and \eqref{eq3}, this contradicts Proposition \ref{t_or_t} for $\Z$ and $\Z'$.
Hence, by \eqref{neigbors_of_s}, we have the following contradiction $d_H(s)-3 = d_{H'}(s) \leq d_{H'}(s,X_I) + d_{H'}(s,Y_I) \leq d_{H_{u,v}}(s,X_I) + d_{H_{t,w}}(s,Y_I) \leq (\frac{d_{H_{u,v}}(s)}{2}-1) + (\frac{d_{H_{t,w}}(s)}{2}-1) = d_H(s)-4$.
\end{proof}
Suppose there exists no $t$-star obstacle at $s$ in $H$. Hence, by Claim \ref{t_star_in_G}, there exists an admissible pair $(st,sw)$ in $H$. By Claim \ref{t_equal_t}, if there exists an obstacle in $H_{t,w}$, then it is a $t$-star obstacle $(t,\C')$.
By $t \notin \{u,v\}$ and \eqref{eq1} in $H_{u,v}$, $d_H(s,t)$ is odd. Hence, by \eqref{eq1} in $H_{t,w}$, $w=t$. Thus $(t,\C')$ is a $t$-star obstacle in $H$, a contradiction.
\end{proof}



Now we are in the position to prove our main result that characterizes the existence of a complete admissible splitting-off. 

\begin{theo} \label{splitting_th}
 Let $H=(V+s,E)$ be a $(2k,k)$-connected graph in $V$ with $k \geq 2$ such that $d_H(s) \geq 4$ is even. There is a complete admissible splitting-off at $s$ if and only if there exists no obstacle at $s$.
\end{theo}
\begin{proof} Suppose there exists an obstacle $(t,\C)$ at $s$. By (\ref{eq1}), every sequence of $\frac{1}{2}d_H(s)$ splitting-off of disjoint pairs at $s$ contains a pair $(st,su)$ with $u \in N_H(s) \setminus \{t\}$. As we noticed after the definition of an obstacle, such a pair is not admissible. Hence there exists no admissible complete splitting-off at $s$.
\medskip

 Now, we prove, by induction on $d_H(s)$, that if there exists no obstacle at $s$, then there exists an admissible complete splitting-off at $s$. 
 Suppose $d_H(s)=4$ and there exists no obstacle at $s$. By Theorem \ref{firstobstacle}, there exists an admissible splitting-off $(su,sv)$ at $s$. Since the only possible splitting-off in $H_{u,v}$ is admissible, there exists an admissible complete  splitting-off at $s$ in $H$.
 
 Now suppose that the theorem is true for $d_{H'}(s)=2\ell$ and $\ell \geq 2$. Let $H=(V+s,E)$ be a $(2k,k)$-connected graph in $V$ such that $d_H(s)=2\ell + 2 \geq 6$ and there exists no obstacle at $s$. 
 By Theorem \ref{firstobstacle}, there exists an admissible splitting-off $(su,sv)$ at $s$.
 If there exists no obstacle at $s$ in $H_{u,v}$, then, by induction, there exists an admissible complete splitting-off at $s$ and we are done. So we may assume that there exists a $t$-star obstacle at $s$ in $H_{u,v}$.
 Since there exists no obstacle at $s$ in $H$, if case (b) of Lemma \ref{pinching_up} occurs then there exists some admissible pair $(st,sw)$ in $H$ such that there exists no obstacle at $s$ in $H_{t,w}$. Thus, by induction, there exists a complete splitting at $s$ in $H$ and we are done. So we may assume that case (a) of Lemma \ref{pinching_up} occurs and we consider $H_{t,t}$ that is $(2k,k)$-connected in $V$. If there exists an obstacle $(t',\C')$ at $s$ in $H_{t,t}$, for the same reason as above, case (a) of Lemma \ref{pinching_up} occurs. Hence $t=t'$ and $(t,\C')$ is an obstacle in $H$, a contradiction.
\end{proof}

\section{Construction of $(2k,k)$-connected graphs}

In this section we provide a construction of the family of $(2k,k)$-connected graphs for $k$ even. The special case $k=2$ has been previously proved by Jord\'an \cite{Jordan_2006}.
\medskip

We need the following extension of Lemma 5.1 of \cite{Jordan_2006} for $k$ even.
Let $G=(V,E)$ be a $(2k,k)$-connected graph, $s$ a vertex of degree even, $(t,\C)$ and $(t,\C')$ two obstacles at $s$. We say that $(t,\C)$ is a \emph{refinement} of $(t,\C')$ if there exists $\X' \in \C'$ such that $\X \sqsubseteq \X'$ for all $\X \in \C$.  An obstacle that has no proper refinement is called \emph{finest}. 

\begin{lemma} \label{finest_obstacle}
 Let $G=(V,E)$ be a $(2k,k)$-connected graph with $k$ even. Let $s$ be a vertex of degree $2k$ and $(t,\C)$ a finest obstacle at $s$. Let $\X \in \C$, $s'$ a vertex  in $X_I$ of degree $2k$ and $(t',\C')$ an obstacle at $s'$. Then there exists $\X' \in \C'$ such that ${X'}_I \subseteq X_I$.
\end{lemma}
\begin{proof} Suppose that the lemma is false.

Suppose $t' \in X_I$. By assumption there exists $\Y' \in \C'$ such that $Y'_I \setminus X_I \neq \emptyset$. Suppose that $t \notin Y'_I$, then $\o{\X} \sqcap \Y'$ is non-trivial and $|w(\o{\X} \sqcup \Y')|=|\{t,t'\}|=2$. Hence, by \eqref{pairs} and since $\o{\X}$ and $\Y'$ are critical, we have $0+0 \geq \d_G(\o{\X} \sqcup \Y') \geq d_G(s',Y'_I) \geq 1$, a contradiction. Hence, $t \in Y'_I$. Now suppose that $X_O \cup Y'_O \neq V$ that is $\o{\X} \sqcap \o{\Y'}$ is non-trivial. We have $w(\o{\X} \sqcup \o{\Y'}) = |\{t,t'\}|=2$, then, by \eqref{pairs} and since $\o{\X}$ and $\o{\Y'}$ are critical, we have $0+0 \geq d_G(\o{X_O} \cap Y'_O, \o{X_I} \cap Y'_I) + \d_G(\o{\X} \sqcup \o{\Y'}) \geq d_G(s',Y'_I) \geq 1$, a contradiction. Hence, $Y'_O \cup X_O = V$, and, for all $\X' \in \C'-\Y'$, $X'_I \subseteq X_I$, a contradiction. Hence $t' \notin X_I$.

Suppose $t' \neq t$. If $t$ belongs to an element $\Z' \in \C'$, then by \eqref{neigbors_of_s}, $d_G(s')-d_G(s',Z'_I) > 2k - k = \d_G(\X)$. Hence there exists $\Y' \in \C'$ with $Y'_I \cap X_I \neq \emptyset$ and $t \notin \Y'$. Thus $\X \sqcap \Y'$ is non-trivial and $|w(\X \sqcup \Y')|=|\{t,t'\}|=2$. Since $\X$ and $\Y'$ are both critical, by \eqref{pairs} and \eqref{eq1}, $0 + 0 \geq d_G(\o{X_O} \cap Y'_O, X_I \cap \o{Y'_I}) \geq d_G(t',s') \geq 1$, a contradiction. Hence $t=t'.$

By $(2k,k)$-connectivity of $G$, $d_G(s',t) \leq k$. Thus, by \eqref{eq1} and $k$ even, $d_G(s',t) < k$. Hence $d_G(s') - d_G(s',t) > 2k-k = \d_G(\X)$ and there exists $\Y' \in \C'$ with $Y'_I \cap X_I \neq \emptyset$. By $|\C'|\geq3$ and assumption, $\X \sqcup \Y', \o{\X} \sqcap \Y'$  and $\X \sqcap \o{\Y'}=\o{\o{\X} \sqcup \Y'}$ are non-trivial, thus, by Claim \ref{tight_intersection}, $\X \sqcap \Y'$ and $\X \sqcap \o{\Y'}$ are  tight bi-sets with wall $t$. 
Thus, in $\C$, $\X$ can be replaced by the bi-sets among $\X \sqcap \Y'$ and $\X \sqcap \o{\Y'}$ which contain at least one neighbor of $s$ in their inner-set. Hence, $(t,\C)$ is not a finest obstacle at $s$, a contradiction.
\end{proof}

We can now decribe and prove the construction of the family of $(2k,k)$-connected graphs.
We denote by $kK_3$ the graph on $3$ vertices where each pair of vertices is connected by $k$ parallel edges. Note that $kK_3$ is $(2k,k)$-connected and it is the only minimally $(2k,k)$-connected graph on $3$ vertices.

\begin{theo} \label{construction_th}
 A graph $G$ is $(2k,k)$-connected with $k$ even if and only if $G$ can be obtained from $kK_3$ by a sequence of the following operations:
\begin{itemize}
 \item[(a)] adding a new edge,
 \item[(b)] pinching a set $F$ of $k$ edges such that, for all vertices $v$, $d_F(v) \leq k$.
\end{itemize}
\end{theo}

\begin{proof} First we prove the sufficiency, that is these operations preserve $(2k,k)$-connectivity. It is clearly true for (a). Let $G'$ be a graph obtained from a $(2k,k)$-connected graph $G=(V,E)$ by the operation (b) and call $s$ the new vertex. We must show that for every non-trivial bi-set $\X$ of $V+s$, we have $f_{G'}(\X) \geq 2k$.
If $\X$ is a non-trivial bi-set of $V$ then $s \notin X_O$ and, by $(2k,k)$-connectivity of $G$,
$f_{G'}(\X) = \d_{G'}(\X) + k|w(\X)|
\geq \d_G(\X) + k|w(\X)|
= f_G(\X) \geq 2k.$
So, by symmetry of $f_{G'}$, we may assume that $X_I = \{s\}$ or $w(\X)=\{s\}.$
If $X_I = \{s\}$, then, by $d_{G'}(s)=2k$ and $d_F(w(\X))\leq k$, we have 
$f_{G'}(\X) = \d_{G'}(\X) + k|w(\X)|
= d_{G'}(s) -  d_{G'}(s,w(\X)) + k|w(\X)|
= d_{G'}(s) -  d_F(w(\X)) + k|w(\X)|
\geq 2k.$
If $w(\X)=\{s\}$ then $\emptyset \neq X_I \neq V$. Hence, by $(2k,k)$-connectivity of $G$ and $|F| = k$, we have
$f_{G'}(\X) = \d_{G'}(\X) + k|w(\X)|
= d_G(X_I) - d_F(X_I) + k
\geq d_G(X_I) - |F| + k
\geq 2k.$
%
\medskip

To see the necessity, let $G$ be a $(2k,k)$-connected graph with at least $4$ vertices. Note that the inverse operation of (a) is deleting an edge and that of (b) is a complete splitting-off at a vertex $s$ of degree $2k$ such that $d_G(s,v) \leq k$ for all $v \in V$. Note also that these inverse operations must preserve $(2k,k)$-connectivity. Thus we may assume that, on the one hand, $G$ is minimally $(2k,k)$-connected and hence, by Lemma 7 of \cite{Kaneko_Ota2000}, $G$ contains a vertex of degree $2k$, and, on the other hand, for every such vertex $u$, there exists no admissible complete splitting-off at $u$, that is, by Theorem \ref{splitting_th}, there exists an obstacle at $u$.

 We choose in $\{(u,(t,\C),\X):$ $d_G(u)=2k$, $(t,\C)$ a finest obstacle at $u$, $\X \in \C\}$  a triple $(u^*,(t^*,\C^*),\X^*)$ with $\X^*$ minimal for inclusion.
By Lemma 7 of \cite{Kaneko_Ota2000}, there exists a vertex $u'$ of degree $2k$ in $X^*_I$. Then, as we have seen, there exits a finest obstacle $(t',\C')$ at $u'$. By Lemma \ref{finest_obstacle}, there exists $\X' \in \C'$ such that $X'_I \subseteq X^*_I$. Since $X'_I \cup u \subseteq  X^*_I$, the triple $(u',(t',\C'),\X')$ contradicts the choice of $(u^*,(t^*,\C^*),\X^*)$.
\end{proof}
 
We mention that the condition $k$ is even is necessary in Lemma \ref{finest_obstacle} and Theorem \ref{construction_th}. Consider the graph obtained from $K_4$ by adding a new vertex $t$ and $3$ edges between $t$ and each vertex of $K_4$. This graph is minimally $(6,3)$-connected but there exists no complete admissible splitting-off at any of the $4$ vertices of degree $6$. Indeed, if $s,a,b,c$ denote the vertices of degree $6$, then $\{(\{a,t\},\{a\}),(\{b,t\},\{b\}),(\{c,t\},\{c\})\}$, is a $t$-star obstacle at $s$.


\section{Augmentation theorem}

In this section, we answer the following question for $k \geq 2$: given a graph what is the minimum number of edges to be added to make it $(2k,k)$-connected. For $k=1$, that is for $2$-vertex-connectivity, this problem had been already solved by Eswaran and Tarjan \cite{Eswaran_Tarjan}.
\medskip

We shall need the following definitions. Let $G=(V,E)$ be a graph and $k$ an integer. An \emph{$s$-extension} of $G$ is a graph $H=(V+s,E \cup F)$ where $F$ is a set of edges between $V$ and the new vertex $s$. The \emph{size} of an $s$-extension of $G$ is defined by $|F|$.
\medskip

We mimic the approach of Frank \cite{frank_1992} for the augmentation problem: first we prove a result on minimal extensions and then, by applying our splitting-off theorem, we get a result on minimal augmentation.


\begin{lemma} \label{minimal_extension}
 Let $G=(V,E)$ be a graph and $k$ an integer. The minimal size of an $s$-extension of $G$ that is $(2k,k)$-connected in $V$ is equal to $\max \left \{ \sum_{\X \in \mathcal{X}}(2k - f_G(\X)) \right \},$ where $\mathcal{X}$ is a family of non-trivial pairwise innerly-disjoint bi-sets of $V$.
\end{lemma}
\begin{proof} If $H'=(V+s,E \cup F')$ is an $s$-extension of $G$ that is  $(2k,k)$-connected in $V$ and $\mathcal{X}'$ is an arbitrary family of non-trivial pairwise innerly-disjoint bi-sets of $V$ then $\sum_{\X' \in \mathcal{X}'}(2k-f_G(\X')) \leq \sum_{\X' \in \mathcal{X}'}(f_H(\X')-f_G(\X')) = \sum_{\X' \in \mathcal{X}'}\d_{(V+s,F')}(\X') \leq |F'|.$ This shows that $\max \leq \min$.
 
 To prove that equality holds, we provide a family $\mathcal{X}$ of non-trivial pairwise innerly-disjoint bi-sets of $V$ and an $s$-extension of $G$ that is  $(2k,k)$-connected in $V$ of size $\sum_{\X \in \mathcal{X}}(2k-f_G(\X))$. We consider the $s$-extension of $G$ whose set of new edges consists of $\max_{\X}(2k-f_G(\X))$ parallel edges $sv$, for each $v \in V$. This extension is obviously $(2k,k)$-connected in $V$. Then we remove as many new edges as possible without destroying the $(2k,k)$-connectivity in $V$. Let us denote by $F$ the set of remaining edges and $H=(V+s,E \cup F)$. In $H$, by minimality of $F$, for each $e \in F$, there exists a tight bi-set of $V$ entered by $e$. Let $\mathcal{X}$ be a family of non-trivial tight bi-sets of $V$ such that 
 \begin{eqnarray}
  & \textrm{each edge of $F$ enters at least one element of $\mathcal{X}$ and }  \label{eq_X} \\
  & \sum_{\X \in \mathcal{X}}|X_I| \textrm{ is minimal.} \label{sum_minimal}
 \end{eqnarray}
 \begin{claim} \label{disjoint_inner_sets}
  The elements of $\mathcal{X}$ are pairwise innerly-disjoint.
 \end{claim}
 \begin{proof} 
  Note that, the degree of each tight bi-set $\X$ in $\mathcal{X}$ is at least one thus $|w(\X)| \leq 1.$
  Suppose there exist two distinct elements $\X$ and $\Y$ in $\mathcal{X}$ such $X_I \cap Y_I \neq \emptyset$, that is $\X \sqcap \Y$ is a non-trivial bi-set of $V$.
  
  If $\X \sqcup \Y$ is a non-trivial bi-set of $V$ then, by $(2k,k)$-connectivity in $V$ of $H$, tightness of $\X$ and $\Y$ and Claim \ref{tight_intersection}, $\X \sqcup \Y$ is tight. Since all the edges of $F$ entering $X_I$ or $Y_I$ enters $(\X \sqcup \Y)_I$, the family obtained from $\mathcal{X}$ by substituting $\X \sqcup \Y$ for $\X$ and $\Y$ satisfies \eqref{eq_X} and, by $X_I \cap Y_I \neq \emptyset$, contradicts \eqref{sum_minimal}. So $X_O \cup Y_O = V$.
 
 If $\X \sqcap \o{\Y}$ and $\o{\X} \sqcap \Y$ are non-trivial bi-sets of $V$ then, by $(2k,k)$-connectivity in $V$ of $H$, tightness of $\X$ and $\Y$ and Claim \ref{tight_intersection}, both $\X \sqcap \o{\Y}$ and $\o{\X} \sqcap \Y$ are tight and $d_H(\o{X_O} \cap \o{Y_I},X_I \cap Y_O)=d_H(Y_I \cap X_O, \o{Y_O} \cap \o{X_I})=0$. Hence all the edges of $F$ entering $X_I$ or $Y_I$ enters $(\X \sqcap \o{\Y})_I$ or $(\o{\X} \sqcap \Y)_I$. Thus the family obtained from $\mathcal{X}$ by substituting $\X \sqcap \o{\Y}$ and $\o{\X} \sqcap \Y$ for $\X$ and $\Y$ satisfies \eqref{eq_X} and, by $X_I \cap Y_I \neq \emptyset$, contradicts \eqref{sum_minimal}. So, by symmetry, we may assume that $X_I \subseteq Y_O.$

 We have $N_H(s) \cap X_I \nsubseteq Y_I$ otherwise $\mathcal{X} - \X$ satisfies (\ref{eq_X}) and contradicts the minimality of $\mathcal{X}$. Thus, by $X_I \subseteq Y_O$, $d_H(s,w(\Y)) \geq 1$ and, since $X_O \cup Y_O = V$ and $\Y $ is non-trivial, $w(\X) \setminus Y_O = X_O \setminus Y_O = (X_O \cup Y_O) \setminus Y_O = V \setminus Y_O$ is non-empty. So $|w(\o{X} \sqcup \Y)| \geq 2.$
 
 For the same reason as above, $N_H(s) \cap Y_I \nsubseteq X_I.$ Thus, by $|w(\X)|\leq 1$ and $w(\X) \setminus Y_O \neq \emptyset$, the set $Y_I \setminus X_O = Y_I \setminus X_I$ contains a neighbor of $s$, that is $\o{\X} \sqcap \Y$ is non-trivial. Thus, by  symmetry of $f_H$, tightness of $\X$ and $\Y$ and \eqref{pairs}, we have the following contradiction
  $0 + 0 = (f_H(\o{\X})-2k)+(f_H(\Y)-2k) \geq d_H(X_I \cap Y_O,\o{X_O} \cap \o{Y_I}) \geq d_H(s,w(\Y)) \geq 1.$
 \end{proof}
 By Claim \ref{disjoint_inner_sets}, \eqref{eq_X} and by tightness of the elements of $\mathcal{X}$, we have $|F| = \sum_{\X \in \mathcal{X}} \d_{(V+s,F)}(\X) = \sum_{\X \in \mathcal{X}} (f_H(\X) - f_G(\X)) = \sum_{\X \in \mathcal{X}} (2k - f_G(\X))$.
\end{proof}

The augmentation theorem goes as follows.

\begin{theo} \label{augmentation_th}
 Let $G=(V,E)$ be a graph and $k \geq 2$ an integer. The minimum cardinality $\gamma$ of a set $F$ of edges such that $(V,E \cup F)$ is $(2k,k)$-connected is equal to
  \[\alpha = \left \lceil \frac{1}{2} \max \left \{ \sum_{\X \in \mathcal{X}}(2k-f_G(\X)) \right \} \right \rceil ,\]
 where $\mathcal{X}$ is a family of non-trivial pairwise innerly-disjoint bi-sets of $V$.
\end{theo}

\begin{proof}

 We first prove $\gamma \geq \alpha$. Let $\mathcal{X}$ be a family of non-trivial bi-sets of $V$ such that the elements of $\mathcal{X}$ are pairwise innerly-disjoint. For each $\X \in \mathcal{X}$, we must add at least $2k - f_G(\X)$ new edges entering the bi-set $\X$ when this quantity is positive. Since the elements of $\mathcal{X}$ are pairwise innerly-disjoint, a new edge may enter at most $2$ elements of $\mathcal{X}$. Hence $2\gamma \geq \sum_{\X \in \mathcal{X}}(2k-f_G(\X)).$

 We now prove $\gamma \leq \alpha$.
  By Lemma \ref{minimal_extension}, there exists an $s$-extension $H=(V+s,E \cup F)$ of $G$  that is $(2k,k)$-connected in $V$ and a family $\mathcal{X}$ of non-trivial pairwise innerly-disjoint bi-sets of $V$ such that
 \[|F| = \sum_{\X \in \mathcal{X}} (2k - f_G(\X)) .\] 
 If $|F|$ is odd, then there exists a vertex $u \in V$ such that $d_H(s,u)$ is odd, in this case, let $F'=F \cup su$ otherwise let $F'=F$. So, in the graph $H'=(V \cup s, E \cup F')$, $d_{H'}(s)$ is even.
 Suppose there exists an obstacle $(t,\C)$ at $s$. By Claim \ref{st_removable}, $H'-st$ is $(2k,k)$-connected in $V$. If $H=H'$ this contradicts the minimality of $|F|$.
 Then $d_H(s)$ is odd and $F'=F+su$ for some vertex $u \in V$ such that $d_H(s,u)$ is odd. If $u \in X_I$ for some $\X \in \C$, then we have $f_{H}(\X) = f_{H'}(\X) - 1 = 2k-1$, a contradiction to the $(2k,k)$-connectivity of $H$. Thus, by \eqref{eq4}, $u=t$ and hence $d_{H'}(s,t) = d_H (s,t) + 1$ is even, that contradicts \eqref{eq1}.
 Hence, by Theorem \ref{splitting_th}, there exists an admissible complete splitting-off at $s$ in $H'$. Let us denote by $F''$ the set of edges obtained by this complete splitting-off. Then $(V,E\cup F'')$ is $(2k,k)$-connected and
 \[|F''| = \frac{1}{2}|F'|= \left \lceil \frac{1}{2} |F| \right \rceil = \left \lceil \frac{1}{2} \sum_{X \in \mathcal{X}} (2k - f_G(X)) \right \rceil.\] 
 This proves $\gamma \leq \alpha$.
\end{proof}


 \bibliographystyle{plain}
 \bibliography{../../biblio/biblio}

\end{document}